\documentclass[12pt,a4paper]{amsart}
\usepackage{amsfonts}
\usepackage{amssymb}
\usepackage{amsmath}
\usepackage{amsthm}
\usepackage{cite}
\usepackage{enumitem}
\usepackage{tikz}
\usepackage{subfigure}
\usepackage{MnSymbol}

\theoremstyle{plain}
\newtheorem{thm}{Theorem}

\newtheorem{lem}{Lemma}
\newtheorem{prop}{Proposition}

\newtheorem{bem}{Remark}

\def\XXint#1#2#3{{\setbox0=\hbox{$#1{#2#3}{\int}$} 
     \vcenter{\hbox{$#2#3$}}\kern-.5\wd0}}

\providecommand{\N}{\mathbb{N}}
\providecommand{\R}{\mathbb{R}}

\DeclareMathOperator{\diver}{div}
\DeclareMathOperator{\dist}{dist}

\renewcommand{\qed}{\hfill $\Box$}

\setitemize{itemsep=+2pt}
\setenumerate{itemsep=+2pt}

\begin{document} 

\allowdisplaybreaks

\title[A note on the local regularity of distributional solutions and subsolutions]{A note on the local
regularity of distributional solutions and subsolutions of semilinear elliptic systems}

\author{Rainer Mandel}
\address{R. Mandel \hfill\break
Scuola Normale Superiore di Pisa \hfill\break
I- Pisa, Italy}
\email{Rainer.Mandel@sns.it}
\date{}

\subjclass[2000]{Primary: }
\keywords{}

\begin{abstract}
  In this note we prove local regularity results for distributional solutions and subsolutions of semilinear
  elliptic systems such as  
  \begin{equation*} 
    L_k^m u_k = f_k(x,u_1,\ldots,u_N) \quad\text{in }\R^n\qquad (k=1,\ldots,N)
  \end{equation*}
  where $L_1,\ldots,L_N$ are of divergence-form and $n\geq 2m$. We show that distributional subsolutions are
  locally bounded from above if $|f_k(x,z)|\leq C(1+|z|^p)$ for $1\leq p<\frac{n}{n-2m},k=1,\ldots,N$. 
  Furthermore, regularity properties of subsolutions and improved versions for bounded subsolutions
  are given. Even for $f_1=\ldots=f_N=0$ our results are new.
\end{abstract}

\maketitle

\section{Introduction}
  
  The starting point of this paper is the following question: Being given a semilinear system of elliptic
  partial equations of the form
  \begin{equation} \label{Gl eq}
    L_k^m u_k = f_k(x,u_1,\ldots,u_N) \quad\text{in }\Omega\qquad (k=1,\ldots,N)
  \end{equation}
  on some open set $\Omega\subset\R^n$ what is the maximal regularity of an arbitrary
  distributional subsolution or solution of \eqref{Gl eq}? More specifically we are interested in assumptions
  on divergence-form operators $L_1,\ldots,L_N$ of second order and nonlinearities $f_1,\ldots,f_N$ which
  ensure that distributional solutions or subsolutions (see \eqref{Gl def distributional solution},\eqref{Gl
  def distributional subsolution} for a definition) are locally bounded or bounded from above, respectively.
  
  \medskip 
  
  The study of unbounded weak and distributional solutions was initiated about 50 years ago and it gave rise
  to several interesting methods and results. Let us try to give a short overview of the subject with a focus
  on second order equations (i.e. $m=1,N=1$) such as
  \begin{equation} \label{Gl Modell-Gl}
     -\diver(A\nabla u) = f(x,u)\quad\text{in }\Omega
  \end{equation}
  where $\Omega$ is a bounded domain in $\R^n,\, n\geq 3$. We start with unbounded solutions that owe
  their existence to the roughness of the matrix function $A$ which we will always
  assume to be positive definite and bounded  on $\Omega$. In a fundamental
  paper Serrin \cite{Ser_pathological} provided explicit examples for matrix functions $A$ and unbounded
  solutions of \eqref{Gl Modell-Gl} for $f\equiv 0$ such that the solutions lie in Sobolev spaces
  $W^{1,s}_{loc}(\Omega)$ with $s<2$.
  Given the fact that local regularity results of De~Giorgi-Nash-Moser type (see chapter~8 in
  \cite{GiTr_elliptic}) are valid for weak solutions lying in $W^{1,2}_{loc}(\Omega)$ we see that different a priori assumptions
  on the regularity of the solution may lead to different kinds of solutions. In the works of Brezis
  \cite{Bre_On_a_conjecture} and Jin, Maz'ya, van Schaftingen \cite{JiMavaS_pathological} this issue was
  analyzed further. For instance, they proved that if the entries of $A$ are only assumed to be continuous
  then solutions lying in $W^{1,1}_{loc}(\Omega)$ do in general not possess better regularity properties whereas
  solutions in $W^{1,p}_{loc}(\Omega),\,p>1$ lie in $W^{1,q}_{loc}(\Omega)$ for all $q\in [1,\infty)$. In
  particular these solutions are always bounded (in contrast to their gradients, 
  see \cite{JiMavaS_pathological}). Under the assumption of H\"older-continuous coefficients the latter result
  had previously been obtained by Hager and Ross \cite{HaRo_a_regularity}.
  For weak solutions in $W^{1,1}_{loc}(\Omega)$ Brezis proved the $W^{1,q}_{loc}(\Omega)$-regularity for
  matrix functions $A$ with Dini-continuous entries, see Theorem~2 in \cite{Bre_On_a_conjecture}. All these
  results concern weak solutions whereas Bao and Zhang \cite{ZhBaI_Regularity_of,ZhBaII_Regularity_of} studied
  regularity results for distributional solutions with Lipschitz continuous $A$. They
  showed that in this case distributional solutions of $-\diver(A\nabla u)=f$ have the ''typical'' regularity
  properties of elliptic problems and no pathological solutions like the ones mentioned above can exist.
  
%
%
  \medskip  
   
  Another way of producing unbounded solutions of \eqref{Gl Modell-Gl} is to consider nonlinearities $f(x,z)$
  that grow sufficiently fast to infinity as $z$ tends to infinity. This can be illustrated via the
  model equation $-\Delta u = u^p$ on a domain $\Omega\subset\R^n$. 
  For $p>\frac{n}{n-2}$ unbounded distributional solutions are given by the explicit formula 
  $\hat u(x):=c_{n,p}|x-x_0|^{-2/(p-1)}$ for some $c_{n,p}>0,x_0\in\Omega$. More sophisticated unbounded
  distributional solutions of this equation in suitable domains
  $\Omega\subset\R^n$ with appropriate boundary conditions are due to Pacard
  \cite{Pac_existence_de_solutions,Pac_existence_and} and Mazzeo,Pacard \cite{MaPa_a_construction} for 
  exponents $p\geq \frac{n}{n-2}$. Finally let us mention that similar constructions were performed in
  \cite{MaRe_Distributional} in the context of nonlinear Schr\"odinger
  equations on $\R^n$ for $p$ slightly larger than $\frac{n}{n-2}$.
   
  \medskip
  
  Let us now describe in which way our results contribute to the issue. We deal with subsolutions
  (Theorem~\ref{Thm subsolution}) and solutions (Theorem~\ref{Thm solution}) of the $2m$-th order semilinear
  elliptic system~\eqref{Gl eq} on an open set $\Omega\subset\R^n$. We restrict our attention to the case
  $n\geq 2m$ which, from the point of view of regularity theory, is more interesting. We find regularity properties
  of subsolutions which will be shown to be optimal in a general setting. A new feature of our approach is
  that these results can however be improved once we add some integrability assumption on the negative parts
  of the subsolutions. Furthermore, even in the easiest case of linear problems of second order equations
  ($m=1,N=1,f=0$) our results are new since the involved linear divergence-form differential operators $L_1,\ldots,L_N$ may have Lipschitz continuous but also less regular coefficient functions, see assumption $(A1)_\alpha$ below.
  In particular we can treat more general situations than in \cite{ZhBaI_Regularity_of,ZhBaII_Regularity_of}
  where distributional solutions of ${-\diver(A\nabla u)=f}$ were investigated under the assumption that $A$
  is locally Lipschitz. Our assumptions on the coefficient functions will be shown to be sharp in the sense
  that for slightly less regular coefficients our regularity results cannot hold any more in view of the
  pathological solutions found by Jin, Maz'ya, van Schaftingen \cite{JiMavaS_pathological} and Serrin
  \cite{Ser_pathological}. In Remark~\ref{Bem 0}~(c) this aspect will be explained in detail.  
  
  \medskip

  Before coming to the statement of our main result let us provide the definitions of distributional solutions
  and subsolutions of \eqref{Gl eq}. To this end we introduce a class of differential operators which is
  suitable for the definition of distributional solutions lying in $L^\alpha_{loc}(\Omega;\R^N)$. We will
  always assume that the following hypothesis is satisfied:
  \begin{itemize}
    \item[$(A1)_\alpha$] The differential operators $L_1,\ldots,L_N$ are given by
     $L_k\phi:= -\diver(A_k\nabla\phi)$ for $\phi\in
     C_c^\infty(\Omega)$ where the matrix functions $A_1,\ldots,A_N$ are symmetric and
     positive definite with entries in $W^{2m-1,\alpha/(\alpha-1)}_{loc}(\Omega)$ for some
    $\alpha\in [1,\frac{n}{n-1})$.
  \end{itemize}
  The assumption $1\leq \alpha<\frac{n}{n-1}$ makes sure that $A_1,\ldots,A_N$ are H\"older-continuous by
  Sobolev's embedding theorem so that several classical results in the theory of elliptic partial differential
  equations \cite{GiTr_elliptic} can be applied in the sequel. In Remark \ref{Bem 0}~(c) we will show that this
  assumption is not only helpful from a technical point of view but also essential for our Theorem to be true.
  We say that $u\in L^\alpha_{loc}(\Omega;\R^N)$ is a distributional solution of \eqref{Gl eq} if we have
  $f(\cdot,u)\in L^1_{loc}(\Omega;\R^N)$ as well as
  \begin{equation} \label{Gl def distributional solution}
    \int_\Omega u_k L_k^m \phi  = \int_\Omega f_k(\cdot,u)\,\phi
     \qquad\text{for all }\phi\in C_c^\infty(\Omega)\qquad (k=1,\ldots,N). 
  \end{equation}
  Accordingly a distributional subsolution is supposed to satisfy
  \begin{equation} \label{Gl def distributional subsolution}
    \int_\Omega u_k L_k^m \phi  \leq \int_\Omega f_k(\cdot,u)\,\phi
     \qquad\text{for all }\phi\in C_c^\infty(\Omega),\phi\geq 0 \qquad (k=1,\ldots,N). 
  \end{equation}
  Next let us state the main result of this paper.
        
  \begin{thm} \label{Thm subsolution} 
    Let $\Omega\subset\R^n$ be open, $n\geq 2m$ and assume $u\in L^\alpha_{loc}(\Omega;\R^N)$ solves
    \begin{equation*} 
      L_k^m u_k \leq C(g+{u_1^+}^p+\ldots+{u_N^+}^p)\quad \text{in } \Omega \qquad (k=1,\ldots,N)
    \end{equation*}
    in the distributional sense with $u^+\in L^p_{loc}(\Omega;\R^N)$ and where $L_1,\ldots,L_N$ satisfy
    $(A1)_\alpha$ and $g\in L^r_{loc}(\Omega),r>\frac{n}{2m}$. Furthermore assume $1\leq p<\frac{n}{n-2m}$ or
    ${p\geq \frac{n}{n-2m},u^+\in L^q_{loc}(\Omega;\R^N)}$ for some $q>\frac{n(p-1)}{2m}$. Then we have $u^+\in L^\infty_{loc}(\Omega;\R^N)$ and $u\in
    W^{2m-1,t}_{loc}(\Omega;\R^N)$ for all $t\in [1,\frac{n}{n-1})$. In addition, the following implications
    hold true:
    \begin{itemize}
      \item[(i)] If $n>2m$ and $u_-\in L^{\tilde q}_{loc}(\Omega;\R^N),\tilde q>\frac{n}{n-2m}$, then
      $u\in W^{2m-1,\tilde t}_{loc}(\Omega;\R^N),\, 1\leq \tilde t< \frac{2m\tilde q}{1+(2m-1)\tilde q}$.
      \item[(ii)] If $n=2m$ and $u_-\in L^\infty_{loc}(\Omega;\R^N)$, then $u\in
      W^{2m-1,\frac{2m}{2m-1}}_{loc}(\Omega;\R^N)$.
    \end{itemize}
  \end{thm}
  
  Here we used the notation $u^\pm:=(u_1^\pm,\ldots,u_N^\pm)$ to denote the vector of the positive/negative
  parts of the component functions of $u$. The proof of Theorem~\ref{Thm subsolution} is based on a
  representation formula for subsolutions and a well-known bootstrap procedure that seems to go back to
  Stampacchia in the case $m=1$. In the following remark we discuss extensions of Theorem~\ref{Thm
  subsolution} and why it can not be essentially improved.
  
  \medskip
  
  \begin{bem} \label{Bem 0} ~ 
    \begin{itemize}
      \item[(a)] The iteration scheme from the proof may be slightly modified to prove local boundedness
      results for problems of the kind 
      $$
      L_k^m u_k\leq C(g_1+g_2{u_1^+}^p+\ldots+g_2{u_N^+}^p)\quad \text{in } \Omega \qquad
      (k=1,\ldots,N) 
    $$ 
    where $g_1\in L^{r_1}_{loc}(\Omega),r_1>\frac{n}{2m}$ and $g_2\in L^{r_2}_{loc}(\Omega),r_2>1$. In this
    situation the critical exponent for local boundedness $\frac{n}{n-2m}$ changes to
    $\frac{n}{n-2m}(1-\frac{1}{r_2})$. Also the differential operators $L_k^m=L_k\circ \ldots \circ L_k$ may
    be replaced by compositions of $m$ different divergence-form operators as in $(A1)_\alpha$ without
    changing the result.
      \item[(b)] The restrictions on $p,q$ from the theorem are optimal in view of several results
      for the model equation $-\Delta u = u^p$. Unbounded distributional solutions in the case
      $p=\frac{n}{n-2}$ were found by Pacard \cite{Pac_existence_and} and in the case $p>\frac{n}{n-2}$
      the function $\hat u$ from the introduction may be taken. Note that 
      $\hat u\in L^q_{loc}(\R^n)$ only for $1\leq q<\frac{n(p-1)}{2}$.
  \item[(c)] The condition $\alpha\in [1,\frac{n}{n-1})$ from assumption $(A_1)_\alpha$ is sharp in the sense
  that Theorem~\ref{Thm subsolution} is not true for $\alpha\in [\frac{n}{n-1},\infty)$.
  Indeed, in \cite{JiMavaS_pathological} Proposition~1.2 the authors
  constructed a cotinuous matrix function $A$ and a weak (and hence distributional) solution $u\in
  W^{1,1}_{loc}(\R^n)\subset L^{n/(n-1)}_{loc}(\R^n)$ of the equation $\diver(A\nabla u)=0$ such that
  $u\notin W^{1,p}_{loc}(\R^n)$ for any other $p>1$. Having a look at the proof of this result (and in
  particular Lemma~2.1 and equation (8) in \cite{JiMavaS_pathological}) one realizes that the coefficient
  functions of $A$ lie in $W^{1,n}_{loc}(\R^n)$. Hence, this example shows that our theorem does not hold in
  the case $\alpha=\frac{n}{n-1}$. For $\alpha\in (\frac{n}{n-1},\frac{n}{n-2})$  Theorem~\ref{Thm
  subsolution} can not hold either since Serrin's pathological solutions from \cite{Ser_pathological} provide
  counterexamples. In the case $\alpha\in [\frac{n}{n-2},\infty)$ the nonlinearity allows us to take the
  solutions from (b) as counterexamples. It is unclear to the author, however, whether the linear problem
  $\diver(A\nabla u)=0$ admits unbounded distributional solutions for such $\alpha$. Finally, we
  think that the case $\alpha=\infty$ has not been treated yet.
      \item[(d)] Also the results from the parts are (i),(ii) are close to optimal. First of all
      one should notice that regularity results in Sobolev spaces of order $2m$ or higher can in general not
      hold since the fundamental solution of $(-\Delta)^m$ does not have such weak derivatives.
      Moreover, the functions $u_\delta(x):= \sigma(x_1^2+\ldots+x_{2m}^2)^{\delta/2}$ for small
      $\delta>0$ and appropriate choice of $\sigma\in\{-1,+1\}$ define bounded polysubharmonic
      functions which do not lie in $W^{2m-1,2m/(2m-1-\delta)}_{loc}(\R^n)$ so that our result for
      $\tilde q=\infty$ may be considered as sharp. More generally, for suitable $\sigma\in\{-1,+1\}$ the
      functions $x\mapsto \sigma(x_1^2+\ldots+x_k^2)^{\delta/2}$ define polysubharmonic functions in
      $L^{\tilde q}_{loc}(\R^n),\frac{n}{n-2m}<\tilde q<\infty$ provided $\delta>\max\{2m-k,-\frac{k}{\tilde
      q}\}$. These functions do not lie in $W^{2m-1,k/(2m-1-\delta)}$ and for certain $\tilde q$ the exponent
      $\frac{k}{2m-1-\delta}$ can be very close to $\frac{2m\tilde q}{1+(2m-1)\tilde q}$. More precisely, if
      we could define these functions for all $k:=\frac{2m\tilde q}{\tilde q-1}\in (2m,n)$ (which, in
      general, is not a natural number) then we would obtain the optimality of the exponent. Unfortunately,
      we have to leave open whether non-formal examples exist or not.
    \end{itemize}
  \end{bem}
    
  \medskip
  
  Theorem~\ref{Thm subsolution} admits a refined version for distributional solutions of~\eqref{Gl eq} which
  we will formulate below for the sake of completeness. It generalizes the results of
  Bao and Zhang \cite{ZhBaI_Regularity_of,ZhBaII_Regularity_of} to nonlinear higher order problems and
  complements the existence and regularity results of Jin, Maz'ya, van Schaftingen
  \cite{JiMavaS_pathological} and Brezis \cite{Bre_On_a_conjecture} in the sense of Remark~\ref{Bem 0}~(c).
  The assumptions on the right hand side $f$ are the following:
  \begin{itemize}
    \item[(A2)] $f:\Omega\times\R^N\to\R^N$ is a Carath\'{e}odory function satisfying 
    \begin{equation*} 
      |f(x,z)| \leq C(g(x)+|z|^p)\quad\text{on }\Omega\times\R^N\quad\text{where }g\in
      L^r_{loc}(\Omega),r\in (1,\infty).
    \end{equation*}
  \end{itemize}
  Under the assumptions $(A1)_\alpha,(A2)$ we show that in the case $1\leq
  p<\frac{n}{n-2m}$ every distributional solution $u\in L^\alpha_{loc}(\Omega;\R^N)\cap
  L^p_{loc}(\Omega;\R^N)$ is a bounded weak solution so that classical elliptic regularity results as in
  \cite{GiTr_elliptic} are applicable. We will omit the proof since it results from discussing special cases
  in the proof of Theorem~\ref{Thm solution}, see Remark~\ref{Bem 1}~(b).
  
  \begin{thm}\label{Thm solution} 
    Let $\Omega\subset\R^n$ be open, $n\geq 2m$ and let $u\in L^\alpha_{loc}(\Omega;\R^N)\cap
    L^p_{loc}(\Omega;\R^N)$ solve
    \begin{equation} \label{Gl equation}
      L_k^m u_k = f_k(x,u_1,\ldots,u_N) \quad\text{in }\Omega \qquad (k=1,\ldots,N)
    \end{equation} 
    in the distributional sense where $L_k,f_k$ satisfy $(A1)_\alpha,(A2)$.     
    Moreover assume $1\leq p<\frac{n}{n-2m}$ or $p\geq \frac{n}{n-2m},u\in L^q_{loc}(\Omega;\R^N)$ for some
    $q>\frac{n(p-1)}{2m}$. Then we have $u\in W^{2m,r}_{loc}(\Omega;\R^N)$. Moreover, $A\in
    C^\infty(\Omega;\R^{n\times n}), f\in C^\infty(\Omega\times\R^N;\R^N)$ implies  $u\in
    C^\infty(\Omega;\R^N)$.
  \end{thm}
    
  \medskip
  
  As a consequence the assumptions of the theorem guarantee that unbounded distributional solutions of
  \eqref{Gl equation} can only exist for $n>2m,p\geq \frac{n}{n-2m}$ and that these solutions have to be
  searched in subspaces of $L^q_{loc}(\Omega)$ with $p\leq q\leq \frac{n(p-1)}{2m}$ while distributional
  solutions lying in higher Lebesgue spaces are automatically locally bounded.
  
  \medskip
  
  In Section~\ref{sec: Lemma} we provide an auxiliary Lemma needed for the proof of the parts (i),(ii) of
  Theorem~\ref{Thm subsolution}. In Section~\ref{sec: Proof} we prove Theorem~\ref{Thm subsolution} using some
  results concerning Green's functions which we provide in the Appendix.

  \section{An auxiliary lemma} \label{sec: Lemma}
  
  For a given ball $B\subset\R^n$, $\gamma<0$ and a Radon measure $\mu$ on
  $B$ with $\mu(B)<\infty$ we introduce the measurable functions $\phi_\gamma,\psi:B\to [0,\infty]$ as
  follows:
  \begin{equation} \label{Gl Def phi_psi}
      \phi_\gamma(x) :=\int_B |x-y|^\gamma\,d\mu(y),\qquad
      \psi(x) := \int_B \log(1/|x-y|)\,d\mu(y).
  \end{equation}
  By definition, a Radon measure on $B$ is a Borel-regular measure on $B$ which is locally finite. 
  In the proof of Theorem~\ref{Thm subsolution} we will need the following result. 
  
  \begin{lem} \label{Lem 1}
    Let $n,k\in\N,\gamma,p,q\in\R$ satisfy $0>\gamma>k-n,q>p\geq 1$ and let $\mu,B$ be given as above. Then
    the following implications hold true:
    \begin{align*}
        \text{(i)}&\quad\text{If }\phi_\gamma\in L^q_{loc}(B) \text{ then }\phi_{\gamma-k} \in L^p_{loc}(B)
        \text{ provided }1\leq  p<\frac{(n+\gamma)q}{n+\gamma+(q-1)k},\\
        \text{(ii)}&\quad \text{If }\psi\in L^\infty_{loc}(B) \text{ then }\phi_{-k} \in
        L^{\frac{n}{k}}_{loc}(B) \text{ provided }n>2k. 
    \end{align*}
  \end{lem}
  \begin{proof}
    In order to prove part (i) we set $r:= \frac{p(q-1)}{q(p-1)}$ so that the following is true:
    $$
      \frac{qr}{p}>r>1,\qquad \gamma-\frac{kr}{r-1}>-n.
    $$ 
    Using these inequalities and H\"older's inequality we obtain for every $\tilde B\subset\subset B$
    \begin{align*}
      \int_{\tilde B} |\phi_{\gamma-k}(x)|^p\,dx
      &= \int_{\tilde B} \Big( \int_B |x-y|^{\frac{\gamma}{r}}\, |x-y|^{\frac{\gamma(r-1)}{r}-k} \,d\mu(y)\Big)^p
      \,dx \\
      &\leq \int_{\tilde B} \Big(\int_B |x-y|^{\gamma}\,d\mu(y)\Big)^{\frac{p}{r}} \, 
       \Big( \int_B |x-y|^{\gamma-\frac{kr}{r-1}}\,d\mu(y)\Big)^{\frac{p(r-1)}{r}}  \,dx \\
      &= \int_{\tilde B} \phi_\gamma(x)^{\frac{p}{r}} \,
      \Big( \int_B |x-y|^{\gamma-\frac{kr}{r-1}}\,d\mu(y)\Big)^{\frac{p(r-1)}{r}}  \,dx \\
      &\leq \|\phi_\gamma\|_{L^q(\tilde B)}^{\frac{p}{r}} \, 
      \Big( \int_{\tilde B} \Big(\int_B |x-y|^{\gamma-\frac{kr}{r-1}} \,d\mu(y) \Big)^{\frac{p(r-1)}{r}\cdot
      \frac{qr}{qr-p}} \,dx\Big)^{\frac{qr-p}{qr}}   \\
      &= \|\phi_\gamma\|_{L^q(\tilde B)}^{\frac{p}{r}} \, 
      \Big( \int_{\tilde B} \int_B |x-y|^{\gamma-\frac{kr}{r-1}} \,d\mu(y) \,dx \Big)^{\frac{qr-p}{qr}}   \\
      &= \|\phi_\gamma\|_{L^q(\tilde B)}^{\frac{p}{r}} \, 
      \Big( \int_B \Big( \int_{\tilde B} |x-y|^{\gamma-\frac{kr}{r-1}} \,dx \Big)\,d\mu(y)
      \Big)^{\frac{qr-p}{qr}}
      \\
      &\leq \|\phi_\gamma\|_{L^q(\tilde B)}^{\frac{p}{r}} \, 
      \Big( \int_B C \,d\mu(y) \Big)^{\frac{qr-p}{qr}}   \\
      &= \|\phi_\gamma\|_{L^q(\tilde B)}^{\frac{p}{r}} \, (C\mu(B))^{\frac{qr-p}{qr}}  
    \end{align*}
    for some $C>0$. This proves part (i). The proof of (ii) is similar. Using analogous
    estimates for $\gamma=0,q=\infty,p=\frac{n}{k},r=\frac{n}{n-k}$ we find that $\phi_{-k}$ lies in $L^p(B)$ provided
    $\psi\in L^\infty(B)$ and 
    $$
      \sup_{y\in B} \int_{\tilde B} |x-y|^{-n}\log(1/|x-y|)^{1-\frac{n}{k}} \,dx < \infty.
    $$
    This condition is satisfied if and only if $n>2k$ which proves part (ii) of the Lemma.
  \end{proof}
  
  Remark~\ref{Bem 0}~(c) tells us that at least for some special choices of $q$ this result can not be
  improved in an essential way.
    
  \section{Proof of Theorem \ref{Thm subsolution}} \label{sec: Proof}

  \medskip
    
  From now on we assume that the hypotheses of Theorem~\ref{Thm subsolution} are satisfied. In particular we
  will use the hypothesis $(A_1)_\alpha$ where $1\leq \alpha<\frac{n}{n-1}$. Without loss of generality, we
  may consider only the case $N=1$ which simplifies the notation in the following. First we discuss the
  regularity of subsolutions of linear equations. We use the following result which, for Lipschitz
  continuous matrix functions in the case $m=1$, is due to Bao and Zhang. We only present the main
  changes with respect to the proof of Theorem~1.3 in~\cite{ZhBaI_Regularity_of} .  
  
  \begin{prop}[Linear regularity I] \label{Prop Linear regI}
    Let $h\in L^\alpha_{loc}(\Omega)$ satisfy $L^m h=0$ in $\Omega$ in the distributional sense. Then $h\in
    W^{2m,t}_{loc}(\Omega)$ for all $t\in [1,\infty)$. In particular $h\in L^\infty_{loc}(\Omega)$.
  \end{prop}
  \begin{proof}
    First we consider the case $m=1$. So let $\beta = \frac{n\alpha}{n-\alpha}>\alpha$ and our first aim is to
    show $h\in L^\beta_{loc}(\Omega)$. To this end we proceed as in \cite{ZhBaI_Regularity_of} and take an
    arbitrary function $w\in C^\infty(\Omega)$, let $B\subset\subset\Omega$ be a ball and let $\phi$ be the
    uniquely determined function satisfying
    \begin{equation} \label{Gl dual equation}
      -\diver(A\nabla \phi) = w\quad\text{in }B,\quad \phi\in W^{1,\frac{\beta}{\beta-1}}_0(B)\cap
       W^{2,\frac{\beta}{\beta-1}}(B),
    \end{equation}
    see Theorem~9.15 in \cite{GiTr_elliptic}. By Theorem~9.19 in \cite{GiTr_elliptic} the function $\phi$ is
    twice continuously differentiable with H\"older-continuous second order derivatives and 
    the differential equation is satisfied almost everywhere. Moreover, we have
    $$
      \|\phi\|_{W^{1,\frac{\alpha}{\alpha-1}}(B)}
      \leq C \|\phi\|_{W^{2,\frac{\beta}{\beta-1}}(B)}
      \leq C' \|w\|_{L^{\frac{\beta}{\beta-1}}(B)}
    $$
    for some $C,C'>0$ by Sobolev's embedding theorem and Lemma~9.17 \cite{GiTr_elliptic}. Therefore
    we may test the equation with $\eta\phi$ for some cut-off function $\eta\in C_c^\infty(B)$ and obtain
    \begin{align*}
      0
      &= \int_B h (-\diver(A\nabla (\eta\phi))) \\
      &= \int_B h (-\diver(A\nabla \phi)\eta -2\nabla\phi^T A\nabla \eta -\diver(A\nabla\eta)\phi) \\
      &\geq \int_B hw\eta -
      C\|h\|_{L^\alpha(B)}\big(\|A\|_{L^\infty(B)}\| |\nabla \phi|\|_{L^{\frac{\alpha}{\alpha-1}}(B)}
      +  \| A\|_{W^{1,\frac{\alpha}{\alpha-1}}(B)}
      \|\phi\|_{L^\infty(B)} \big)  \\
      &\geq \int_B hw\eta -
      C'\|h\|_{L^\alpha(B)}\|A\|_{W^{1,\frac{\alpha}{\alpha-1}}(B)}
      \|\phi\|_{W^{1,\frac{\alpha}{\alpha-1}}(B)}  \\
      &\geq \int_B hw\eta -
      C''\|h\|_{L^\alpha(B)} \|A\|_{W^{1,\frac{\alpha}{\alpha-1}}(B)} 
      \|w\|_{L^{\frac{\beta}{\beta-1}}(B)}
    \end{align*}
    by definition of $\beta$ for some $C,C',C''>0$ independent of $w$. As in \cite{ZhBaI_Regularity_of} the
    dual characterization of Lebesgue spaces gives $h\in L^\beta_{loc}(B)$ and iterating this
    procedure as in \cite{ZhBaI_Regularity_of} yields $h\in L^t_{loc}(B)$ for all $t<\infty$.
    
    \smallskip
    
    In order to prove the existence of weak derivatives we can not proceed as in \cite{ZhBaI_Regularity_of}
    since there a difference quotient method is used which relies on the Lipschitz continuity of
    the matrix function.
    Instead we continue with the duality argument. To this end we consider $w_0,w_1,\ldots,w_n\in C^\infty(B)$
    and choose $\phi$ to be the uniquely determined function satisfying
    \begin{equation} \label{Gl dual equation}
      -\diver(A\nabla \phi) = w_0+ \sum_{j=1}^n \partial_j w_j\quad\text{in }B,\quad \phi\in
      W^{1,2}_0(B)\cap W^{2,2}(B).
    \end{equation}
    As above, $\phi$ is twice continuously differentiable with
    H\"older-continuous second order derivatives and the differential equation
    is satisfied almost everywhere. Instead of Lemma~9.17 in \cite{GiTr_elliptic} we use 
    the estimate
    $$
      \|\phi\|_{W^{1,\frac{t}{t-1}}(B)} 
      \leq C \sum_{j=0}^n \|w_j\|_{L^{\frac{t}{t-1}}(B)} 
    $$
    from Lemma~2.2~(i) \cite{DoMu_estimates_for} for $t>\frac{n}{n-1}>\alpha$. Proceeding as
    above we find $C,C',C''>0$ such that
    \begin{align*} 
      0
      &= \int_B h (-\diver(A\nabla (\eta\phi))) \\      
      &\geq \int_B h\big(w_0+\sum_{j=1}^n \partial_j w_j\big)\eta -
      C\|h\|_{L^t(B)} (\| \|A\|_{L^\infty(B)} \||\nabla \phi|\|_{L^{\frac{t}{t-1}}(B)}
      + \| A\|_{W^{1,\frac{\alpha}{\alpha-1}}(B)} \|\phi\|_{\frac{\alpha t}{t-\alpha}}) \\
      &\geq  \int_B h\big(w_0+\sum_{j=1}^n \partial_j w_j\big)\eta -
      C' \|h\|_{L^t(B)} \| A\|_{W^{1,\frac{\alpha}{\alpha-1}}(B)} \|\phi\|_{W^{1,\frac{t}{t-1}}(B)} \\
      &\geq \int_B h\big(w_0+\sum_{j=1}^n \partial_j w_j\big)\eta -
      C''\|h\|_{L^t(B)} \|A\|_{W^{1,\frac{\alpha}{\alpha-1}}(B)} \sum_{j=0}^n
        \|w_j\|_{L^{\frac{t}{t-1}}(B)}.
    \end{align*}
    In the second inequality we used Sobolev's embedding theorem and $\alpha<\frac{n}{n-1}$.
    From this estimate and the dual characterisation of $W^{1,t}(B)$, see Theorem~3.9 in~\cite{AdFo_Sobolev},
    we arrive at $h\in W^{1,t}(B)$ for all $t>\frac{n}{n-1}$. As in~\cite{ZhBaI_Regularity_of} this induces
    $h\in W^{2m,t}(B)$ for all $t<\infty$. 
    
    \smallskip
    
    In the case $m\geq 2$ instead of \eqref{Gl dual equation} (i.e. $L\phi=w$ in $B$) one solves
    $L^m\phi=w$ in $B$ in some subspace of $W^{2m,\beta/(\beta-1)}(B)$ which can be done by induction over
    $m$ using the same theorems as above. Similar estimates then allow to conclude as in the case $m=1$.
  \end{proof}
  
   With Proposition~\ref{Prop Linear regI} at hand, we may now discuss the regularity properties of
   subsolutions of linear problems. One main feature of our result is that integrability assumptions
   on the negative parts of subsolutions can be used to deduce slightly better regularity properties.
          
  \begin{prop}[Linear regularity II]  \label{Prop Linear regII}
    Let $u\in L^\alpha_{loc}(\Omega)$ satisfy $L^m u\leq g$ in $\Omega$ in the distributional sense
    where $g\in L^r_{loc}(\Omega),g\geq 0$. Then the following implications hold true:
    \begin{itemize}
      \item[(i)] If $r=1$ then $u^+ \in L^s_{loc}(\Omega)$ for all $s\in [1,\frac{n}{n-2m})$,
      \item[(ii)] If $r\in [1,\frac{n}{2m})$ then $u^+ \in L^{nr/(n-2mr)}_{loc}(\Omega)$,
      \item[(iii)] If $r>\frac{n}{2m}$ then $u^+ \in L^{\infty}_{loc}(\Omega)$.
    \end{itemize}
    In each of these cases one has $u\in W^{2m-1,t}_{loc}(\Omega)$ for all $t\in
    [1,\frac{n}{n-1})$. Moreover, assuming $u^-\in L^{\tilde q}_{loc}(\Omega)$ and $r\geq \frac{2mn\tilde
    q}{2mn\tilde q+n-\tilde q(n-2m)}$ the following implications hold true: 
    \begin{itemize}
      \item[(iv)] If $\tilde q>\frac{n}{n-2m},n>2m$ then $u\in W^{2m-1,t}_{loc}(\Omega)$ for all 
      $t\in [1,\frac{2m\tilde q}{1+(2m-1)\tilde q})$,
      \item[(v)] If $\tilde q=\infty,n=2m$ then $u\in W^{2m-1,2m/(2m-1)}_{loc}(\Omega)$. 
    \end{itemize}  
  \end{prop}
  \begin{proof}
    It suffices to verify the above-mentioned regularity properties in an arbitrary compactly contained ball
    $B\subset\subset\Omega$. For any given such ball let $G_1$ be the Green's function of the operator $L$ on
    a slightly larger ball $B'$ associated to homogeneous Dirichlet boundary conditions on $\partial B'$.
    Under our regularity assumption $(A1)_\alpha$ the existence of $G_1$ is guaranteed by Theorem~1.1 in
    \cite{GrWi_the_green}. We define the function $G_m:B'\times B'\to [0,\infty]$ inductively via 
    \begin{equation} \label{Gl Def Gk}
      G_k(x,y)= \int_{B'} G(x,z)G_{k-1}(z,y)\,dy \qquad (k=2,\ldots,m).
    \end{equation}
    Then one has $L^m G_m(\cdot,y) = \delta(\cdot-y)$ in $B$ in the distributional sense as well as
    \begin{align} \label{Gl properties Gm}
      \begin{aligned}
      |\partial^\alpha_x G_m(x,y)| 
      &\leq C |x-y|^{2m-n-|\alpha|} \qquad\text{for } |\alpha|=1,\ldots,2m-1
      &&(x,y\in B), \\
      G_m(x,y)
      &\leq \begin{cases}
        C|x-y|^{2m-n} &, \text{ if }n>2m, \\
        C\log(1/|x-y|)+C' &, \text{ if }n=2m
      \end{cases} 
	  &&(x,y\in B), \\
	  G_m(x,y)
	  &\geq 
	  \begin{cases}
        c|x-y|^{2m-n} &, \text{ if }n>2m \\
        c\log(1/|x-y|)-c' &, \text{ if }n=2m
      \end{cases}       
	  &&(x,y\in B).
	  \end{aligned}
    \end{align}
    Here, $c,c',C,C'$ are positive numbers independent of $x,y\in B$, $\alpha\in\N_0^n$ is a multiindex and
    $\delta$ is the Dirac measure on $\R^n$ centered at $0$. In the Appendix we provide the references for
    these estimates.
     
    \smallskip
    
    \noindent
    {\it Step 1: A representation formula.}\; The function $v_B:= -u|_B+\int_B G_m(\cdot,y)g(y)\,dy$ satisfies
    ${L^mv\geq 0}$ in~$B$ in the distributional sense. Using Theorem~2.17 in~\cite{Mitrea_distributions} 
    we obtain $L^mv_B = \mu_B$ in~$B$ where $\mu_B$ is a Radon measure. Defining $h_B:= -v_B+\int_B
    G_m(\cdot,y)\,d\mu_B(y)$ we arrive at
    \begin{align} \label{Gl polyharmonic Darstellung}
      \begin{aligned}
      u|_B
      &= - \underbrace{\int_B G_m(\cdot,y)\,d\mu_B(y)}_{=:u_1} 
        + \underbrace{\int_B G_m(\cdot,y)g(y)\,dy}_{=:u_2} 
        + h_B\quad
      \text{where }L^mh_B=0, h_B\in L^1_{loc}(B).
      \end{aligned}
    \end{align}
    Moreover, $\mu_B = (-L^mu+g)|_B = ((-L^mu+g)|_{B'})|_B=\mu_{B'} \righthalfcup B$ implies
    $\mu_B(B)<\infty$ since $\mu_B'$ is a Radon measure and thus locally finite.
    This property will be used when we apply Lemma~\ref{Lem 1} in Step~3.
    
    \smallskip
    
    \noindent
    {\it Step 2: Proof of (i),(ii),(iii) -- Integrability.}\; By  the Hardy-Littlewood-Sobolev Theorem (see
    \cite{LiLo_Analysis}, Theorem~4.3) $u_2$ has the integrability properties
    which we claimed to hold for $u$ in (i),(ii),(iii), respectively. By Proposition~\ref{Prop
    Linear regI} we moreover have $h_B\in L^\infty_{loc}(B)$. Hence, the inequality $u|_B^+\leq u_2+|h_B|$
    implies that $u^+$ lies in the same Lebesgue spaces as $u_2$ which is what we wanted to show.
    
    \smallskip
    
    \noindent
    {\it Step 3: Proof of (iv),(v) -- Regularity.}\;  From \eqref{Gl properties Gm} we get $u_1,u_2\in
    W^{2m-1,t}(B)$ for all $t\in [1,\frac{n}{n-1})$. Hence, Proposition~\ref{Prop Linear regI} tells us that $u=-u_1+u_2+h_B$ lies
    in the same spaces. Now let us additionally assume $u^-\in L^{\tilde q}(B),g\in L^r_{loc}(B)$ for
    $\tilde q,r$ as in the statement of the theorem. The assumption on $r$, the upper bounds for the
    derivatives of $G_m$ from \eqref{Gl properties Gm} and the Hardy-Littlewood-Sobolev Theorem imply
    \begin{equation} \label{Gl regI}
      u_2 \in W^{2m-1,\frac{nr}{n-r}}_{loc}(B) 
      \subset W^{2m-1,\frac{2m\tilde q}{1+(2m-1)\tilde q}}_{loc}(B). 
	\end{equation}
    Furthermore, the assumption $u^-\in L^{\tilde q}_{loc}(\Omega)$ implies $u_1+|h_B|\in L^{\tilde q}(B)$ via 
    \eqref{Gl polyharmonic Darstellung} and thus ${u_1\in L^{\tilde q}_{loc}(B)}$ by 
    Proposition~\ref{Prop Linear regI}. In the case $n>2m$ this implies
    $\phi_{2m-n}\in L^{\tilde q}_{loc}(B)$ where $\phi_{2m-n}$ was defined in \eqref{Gl Def phi_psi}. Indeed,
    using the lower bound for $G_m$ from \eqref{Gl properties Gm} we get 
    $$
         u_1(x) 
         = \int_B G_m(x,y)\, d\mu_B(y) 
         \geq c \int_B |x-y|^{2m-n}\,d\mu_B(y)
         = c\, \phi_{2m-n}(x).
    $$ 
    From Lemma \ref{Lem 1} (i) we obtain $\phi_{1-n}\in L^p(B)$ for $p\in [1,\frac{2m\tilde
    q}{1+(2m-1)\tilde q})$ so that the upper bound for the derivatives of $G_m$ from \eqref{Gl properties Gm}
    imply 
    \begin{equation} \label{Gl regII}
      u_1\in W^{2m-1,p}_{loc}(B) \qquad\text{if } p\in [1,\frac{2m\tilde q}{1+(2m-1)\tilde q}),\; n>2m.
    \end{equation} 
    Similarly, in the case $n=2m$ the assumption $u^-\in L^\infty_{loc}(\Omega)$ implies $u_1,\psi\in
    L^\infty_{loc}(B)$ so that Lemma~\ref{Lem 1}~(ii) yields $\phi_{1-n}=\phi_{1-2m}\in
    L^{2m/(2m-1)}_{loc}(B)$ and thus
    \begin{equation} \label{Gl regIII}
       u_1\in W^{2m-1,\frac{2m}{2m-1}}_{loc}(B) \qquad\text{if }\tilde q=\infty,\; n=2m.
    \end{equation} 
    From \eqref{Gl regI},\eqref{Gl regII},\eqref{Gl regIII} and Proposition~\ref{Prop Linear regI} we obtain 
    the assertion of the theorem.
  \end{proof}

  \begin{bem} \label{Bem 1}~
    \begin{itemize}
      \item[(a)] Every (nonnegative) Radon measure $\mu$ on $B$ defines a distributional
      subsolution via the formula \eqref{Gl polyharmonic Darstellung}.
      \item[(b)] In the more special case $L^m u = g$ we have $d\mu_B(y)=g^-(y)\,dy$ which improves
      the regularity properties of $u_1$ according to the integrability properties of $g^-$. This is
      responsible for the fact that solutions of $L^m u=g$ with $g\in L^r_{loc}(\Omega),r\geq 1$ are
      more regular than subsolutions. For elliptic equations with measure-valued right hand sides (such as
      subsolutions) this is in general not true.
    \end{itemize}
  \end{bem}

  \medskip
  
  \noindent
  {\it Proof of Theorem \ref{Thm subsolution}} 
  
  \medskip
  
  \noindent
  Applying Proposition \ref{Prop Linear regII} to $g:=C(1+{u^+}^p)$ we find that it is sufficient to
  prove ${u^+\in L^\infty_{loc}(\Omega)}$. In case $1\leq p<\frac{n}{n-2m}$ the proposition yields $u^+\in
  L^{q_0}_{loc}(\Omega)$ and thus $g\in L^{q_0/p}_{loc}(\Omega)$ for some $q_0\in (p,\frac{n}{n-2m})$. Using
  again Proposition~\ref{Prop Linear regII} one inductively proves $u^+\in L^{q_k}_{loc}(\Omega)$  where
  \begin{equation*}
    q_{k+1} :=
      \begin{cases}
  	  \frac{nq_k}{np-2mq_k} &, \text{ if } 1<\frac{q_k}{p}<\frac{n}{2m}, \\
	  2q_k &, \text{ if } \frac{q_k}{p} = \frac{n}{2m},  \\
	  \infty &, \text{ if } \frac{q_k}{p} >\frac{n}{2m}.
      \end{cases} 
  \end{equation*}
  The sequence $(q_k)$ increases due to $q_0>\frac{n(p-1)}{2m}$ (because of $q_0>p$) and reaches $+\infty$
  after finitely many steps so that the assertion is proved in the case $1\leq p<\frac{n}{n-2m}$.
  In the case $n>2m,p\geq \frac{n}{n-2m}$ the assumption $u^+\in L^q_{loc}(\Omega),q>\frac{n(p-1)}{2m}$ leads
  to the choice $q_0:= q$ and the same iteration as above gives the result. \qed
  
  \section*{Appendix -- On Green's function} \label{sec: Greens function}
  
    Let us briefly recall why the estimates from \eqref{Gl properties Gm} hold if the assumption
    $(A_1)_\alpha$ is satisfied. First we note that due to this
    assumption the $(2m-2)-$th order derivatives of the matrix functions $A_1,\ldots,A_N$ from the theorem are
    H\"older-continuous by Sobolev's embedding theorem so that, generally speaking, $L^p$-estimates and
    Schauder estimates for elliptic problems in divergence-form are applicable. In the following let $A$ be
    one of these matrices and let $G_m$ be the function defined in \eqref{Gl Def Gk}.
    
    \smallskip
    
    We first analyze the properties of $G_1$ and we start with the exceptional case $n=2$.
    The logarithmic bounds for $G_1$ are proved in section~6 in \cite{DoMu_estimates_for}. Notice that for the
    lower bounds one uses that $B$ is compactly contained in $B'$. For the upper bounds of
    the derivatives let us recall from \cite{DuzMin_Gradientestimates} and estimate~(2) in~\cite{DoMu_estimates_for} that the
    following interior estimates hold for any weak solution $w$ of $\diver(A\nabla w)=0$ in $B'$:
    \begin{equation} \label{Gl int_est n=2}
      \|\nabla^k w\|_{L^\infty(B_d)}\leq Cd^{-k} \|\nabla w\|_{L^{2,\infty}(B_{2d})} \qquad (1\leq k\leq
      2m-1,n=2)
    \end{equation} 
    Here, $B_{2d}$ is a ball of radius $2d$ contained in $B'$ and $L^{2,\infty}(B_{2d})$ denotes the
    Lorentz space (or weak-$L^2$ space). We refer the interested reader to the paper of Dolzmann and M\"uller
    \cite{DoMu_estimates_for} where these spaces are used in the context of Green's functions.
    From estimate~(11) in \cite{DoMu_estimates_for} we get $\|\nabla G_1(\cdot,y)\|_{L^{2,\infty}(B')}\leq C$
    so that the interior estimates \eqref{Gl int_est n=2} applied to $G_1(\cdot,y)$ on the ball $B_{2d}(x)$
    with $d=\frac{1}{4}\min\{|x-y|,\dist(B,\partial B')\}$ yield
    \begin{align*}
      |\partial^\alpha_x G_1(x,y)| 
      \leq C |x-y|^{2-n-|\alpha|} \qquad (x,y\in B,\;1\leq |\alpha|\leq 2m-1,\;n=2), 
    \end{align*}
    see also \cite{DoMu_estimates_for}, Lemma~3 where the same reasoning was used.
    This proves the estimates for $G_1$ from \eqref{Gl properties Gm} for $n=2$.
    
    \smallskip
    
    In the case $n\geq 3$ the bounds for $G_1(x,y)$ follow from \cite{GrWi_the_green}, Theorem~1.1. The
    bounds for the derivatives of $G_1(\cdot,y)$ follow from interior estimates in the same manner as above.
    Instead of \eqref{Gl int_est n=2}, however, one uses 
    $$
      \|\nabla^k w\|_{L^\infty(B_d)}\leq Cd^{-k}\|w\|_{L^\infty(B_{2d})} 
      \qquad (k\leq 2m-1,n\geq 3)
    $$ 
    This Schauder type estimate follows inductively from \cite{GiTr_elliptic}, Corollary~6.3, see also
    Lemma~3.1 in \cite{GrWi_the_green} for the special case $k=1$. Hence we obtain
    \begin{align*}
      |\partial^\alpha_x G_1(x,y)| 
      \leq C |x-y|^{2-n-|\alpha|} \qquad
      (x,y\in B,\;|\alpha|\leq 2m-1,\;n\geq 3) 
    \end{align*}
	so that the estimates for $G_1$ in the case $n\geq 3$ are proved, too.
	
	\smallskip
	
	Finally, using the estimates for $G_1$ one can inductively derive the corresponding estimates for $G_m$
	via the formula \eqref{Gl Def Gk}.

  \section*{Acknowledgements} 
    The work on this project was supported by the Deutsche Forschungsgemeinschaft 
    (DFG, German Research Foundation) - grant number MA 6290/2-1.

\medskip
  
\bibliographystyle{plain}
\bibliography{doc}

\end{document}